\documentclass[a4paper]{article}
\usepackage{amsmath,amsfonts,amssymb}
\usepackage{graphicx, epsfig}
\newtheorem{theorem}{Theorem}[section]

\def\RR{{\Bbb R}} %reals
\def\CC{{\Bbb C}} %complexes
\newenvironment{proof}[1][Proof]{\textbf{#1.} }
{\hfill\rule{0.5em}{0.5em}\medskip}
\newenvironment{proof*}[1][Proof]{\textbf{#1.} }{}

\def\epsilon{\varepsilon}

\begin{document}

%%%%%%%%%%%%%%%%%%%%%TITLE
%%%%%%%%%%%%%%%%%%%%%%%%%%

\title{A little scholium on Hilbert-Rohn\\
via the total reality of  $M$-curves:\\
Riemann's flirt with Miss Ragsdale}

%\title{A little scholium on Hilbert-Rohn theorems\\
%and the total reality of  $M$-curves:\\
%Riemann's flirt with Miss Ragsdale?}

\author{Alexandre Gabard}
\maketitle

\newbox\quotation
\setbox\quotation\vtop{\hsize
%%6.8cm
7.8cm \noindent

\footnotesize {\it Poincar\'e hat zuerst die Frage nach dem
Gesamt\-verlauf der reellen L\"osungen von
Differentialgleich\-ungen mit topo\-lo\-gischen Mitteln
behandelt.}

\noindent Hellmuth Kneser, 1921, in {\it Kurvenscharen auf
geschlossenen Fl\"achen {\rm \cite{Kneser_1921}}}.

%\smallskip
%{\it Ich musste die ver\"ucktesten Kurvenscharen integrieren.}

%\noindent Oswald Teichm\"uller writing to his Ph. D. student.

%\smallskip
%{\it Is there a life without a metric?}

%\noindent Misha Gromov in Spaces and questions, 1999.
}

\hfill{\hbox{\copy\quotation}}

\medskip

\newbox\abstract
\setbox\abstract\vtop{\hsize 12.2cm \noindent

%%%%%%%%%%%CLEAN ABSTRACT (ARXIV)

\noindent\textsc{Abstract.} This note presents an elementary proof
of Hilbert's 1891 Ansatz of nesting for $M$-sextics, along the
line of Riemann's Nachlass 1857 and a simple Harnack-style
argument (1876). Our proof seems to have escaped the attention of
Hilbert (and all subsequent workers) [but alas turned out to
contain a severe gap, cf. Introduction for more!]. It uses a bit
Poincar\'e's index formula (1881/85). The method applies as well
to prohibit Rohn's scheme $\frac{10}{1}$, and therefore all
obstructions of Hilbert's 16th in degree $m=6$ can be explained
via the method of total reality. (The same ubiquity of the method
is conjectured in all degrees, and then suspected to offer new
insights.) More factually, a very simple and robust phenomenon of
total reality on $M$-curves of even order is described (the
odd-order case being already settled in Gabard 2013), and it is
speculated that this could be used as an attack upon the (still
open) Ragsdale conjecture for $M$-curves (positing that $\vert
\chi\vert\le k^2$). Of course a giant gap still remains to be
bridged in case the latter conjecture is true at all. Alas, the
writer has little experimental evidence for the truth of the
conjecture, and the game can be a hazardous one. However we
suspect that the method of total reality should at least be
capable of recovering the weaker Petrovskii bound, or strengthened
variants due to Arnold 1971. This text has therefore merely
didactic character and offers no revolutionary results, but tries
to reactivate a very ancient method (due basically to Riemann
1857) whose swing seems to have been somewhat underestimated, at
least outside of the conformal-mapping community.}

\centerline{\hbox{\copy\abstract}}

\iffalse
\bigskip

%2000 {\it Mathematics Subject Classification.} {\rm 57N99,
%57R30, 37E35.}

\bigskip\fi

%\normalsize

{\small \tableofcontents}

\section{Introduction}\label{sec1}

[20.04.13] The flirt suggested in our title is a fictional one,
which cannot have occurred between Bernhard Riemann (1826--1866)
and Virginia Ragsdale (1870--1945). However when the latter came
to visit Klein and Hilbert in G\"ottingen (ca. 1903?) the spirit
of Riemann was most vivid than ever and we shall try to speculate
about a direct connection between the works of both scientists.

In a previous paper (Gabard 2013 \cite{Gabard_2013_Ahlfors}), we
made an essay to connect a certain theory of {\it total reality\/}
rooted in Riemann's work on conformal representation with
Hilbert's 16th on real algebraic curves by extrapolating a bit
(hopefully not fallaciously)
%along
the eclectic visions of V.\,A.~Rohlin 1978. We do not repeat here
the vast body of knowledge and array of conjectures accumulated in
both
%%%fields
disciplines and hope to have made there sufficiently explicit a
possible deep interpenetration of both topics. The aim of this
note is to illustrate the method of total reality on a more
%%very basic
concrete terrain, namely {\it Hilbert's nesting Ansatz\/} for
$M$-sextics which
%can be considered as
is undeniably the first nontrivial result (1891)
%along
paving the way toward the general formulation (in 1900) of
Hilbert's 16th (isotopic classification of real plane algebraic
curves, i.e. how ovals of such curves are distributed among
themselves, nested and mutually positioned).
%%%%prose from Wilson's abstract

{\it Bibliographical references.}---To keep the bibliography of
the present text within reasonable limits,  whenever a work is
%%%loosely
cited by specifying only its author name and date (of publication)
we refer the interested reader to the extensive bibliography
compiled in Gabard 2013 \cite{Gabard_2013_Ahlfors}.

{\it Glossary of synonyms}.--- $\bullet$ Harnack(-maximal)
curve=$M$-curve, jargon of Petrovskii 1938, where $M$ stands
probably for maximal.

%%%%%$\bullet$
\medskip

$\bigstar\bigstar\bigstar$ {\it Very Important Warning (Mea Culpa)
(added in proof the [22.04.13])}.---After having posted this note
on the arXiv (yet the day prior to its diffusion), we noticed that
our proofs of the Hilbert and Rohn theorems
(\ref{Hilbert-via-total-reality:thm}) and
(\ref{Rohn-via-total-reality:thm})---via the method of total
reality due to Riemann---contains a serious gap. {\it Exercise:}
detect our mistake without reading the next hint in tiny
calligraphy.

{\tiny

\smallskip

{\it Hint:} it seems that we have overlooked the possible
presences of centers singularities (infinitesimally like
concentric circles) in the foliation (also contributing to
positive indices). Such centers may occur when curves of the
pencil of quartics contract an oval toward a solitary node.

}

\smallskip

If optimistic, this defect can perhaps be repaired, admittedly
after much more efforts. We decided to still publish this note for
two reasons. First, in the hope that someone is able to arrange a
proof of Hilbert (and Rohn) along the method of Riemann. Second,
our main result
(\ref{total-reality-M-curves-EVEN-punching-card:thm}) on the total
reality of $M$-curves is not affected by this issue and
complements the odd-degree case settled in  Gabard 2013
\cite[Thm~31.12, p.\,402]{Gabard_2013_Ahlfors}. Alas, this main
result has very basic character and should merely be regarded as a
first step toward deeper problems \`a la Hilbert, Rohn or Ragsdale
(that we are presently unable to tackle). Therefore, it is evident
that both our title (and abstract) are much immature (not to say
pathetic), but we left them unchanged deliberately in the hope to
attract more qualified workers to the question.
%hope again
%that the situation can be repaired after much more work.
Of course anybody able to complete the programme from
Riemann-to-Hilbert can build upon our free-source file in case its
historical aspects seem of some didactic value.
%(evidently no co-authorship is requested in view of the
%null-value of our contribution).

\subsection{Hilbert's Ansatz: overview of all known proofs
(Hilbert 1891/1900/01, Wright 1907, Kahn 1909, L\"obenstein 1910,
Rohn 1911/13, Donald 1927, Hilton 1936, Petrovskii 1933/38,
Kervaire-Milnor 1961, Arnold 1971, \\ Rohlin 1974/78)}

$\bullet$ {\it Hilbert turning to a geometer.}---In 1891, in a
genius stroke without any antecedents, Hilbert
%arrived at
advanced (without proof, quite uncharacteristic of his style) the
conclusion that a sextic curve which is Harnack-maximal (i.e. with
the maximum number $11$ of ovals\footnote{As pointed by Elias
Boul\'e it seems that 11 is also the number of planets circulating
in the Solar system, when nano-objects like Ceres are included
into the count. Ceres is the greatest, and first detected (1801),
asteroid gravitating somewhere between Mars and Jupiter with a
diameter of about 1000 km.}) cannot have all its
%%%eleven
ovals {\it
unnested\/} lying outside each other. Hilbert 1891
\cite{Hilbert_1891_U-die-rellen-Zuege} confessed in a footnote his
proof to be
%very
exceptionally complicated and highbrow, more precisely:

\smallskip
{\footnotesize

``Diesen Fall $n=6$ habe ich einer weiteren eingehenden
Untersuchung unterworfen, wobei ich\,---\,freilich auf einem
au{\ss}erordentlich umst\"andlichen Wege\,--- fand, da{\ss} die
elf Z\"uge einer Kurve 6-ter Ordnung keinesfalls s\"amtlich
au{\ss}erhalb un voneinander getrennt verlaufen k\"onnen. Dieses
Resultat erscheint mir deshalb von Interesse, weil er zeigt,
da{\ss} f\"ur Kurven mit  der Maximalzahl von Z\"ugen der
topologisch einfachste Fall nicht immer m\"oglich ist.''

}

\smallskip

It required several generations of workers until the method of
Hilbert
%%%became solidified
reached full maturity. The detailed story is probably best
recorded in Gudkov's survey of 1974 \cite{Gudkov_1974/74}, but let
us sketch it briefly (while adding some ``inedited'' items to the
narration). Hilbert himself seems to have been quite fluctuant in
evaluating the level of rigor of his proof. As far as we know,
%Hilbert
he  never published himself a proof, but supervised two
G\"ottingen Dissertations on the question (Kahn 1909
\cite{Kahn_1909} and L\"obenstein 1910 \cite{Löbenstein_1910}),
which apparently turned out to be inconclusive. At least this is
the opinion of both  Rohn 1913 \cite[p.\,178]{Rohn_1913} and
Gudkov 1974 (p.\,41), who actually asserts that on their own
admissions those writers (Hilbert's girls) confessed to have
failed proving  nonexistence of a $C_6$ of unnested type $11$.
Yet, it is slightly puzzling that Hilbert 1909
\cite{Hilbert_1909-Ueber-die-Gestalt-sextic} qualified the proof
of Kahn-L\"obenstein as complete, more precisely:

\smallskip

{\footnotesize

``[\dots] eine ebene Kurve 6-ter Ordnung hervorgehen, die aus elf
au{\ss}erhalb voneinander getrennt verlaufenden Z\"ugen
best\"ande. Da{\ss} aber eine solche Kurve nicht existiert, ist
einer der tiefstliegenden S\"atze aus der Topologie der ebenen
algebraischen Kurven; derselbe ist k\"urzlich von G.~Kahn und
K.~Loebenstein (Vgl. die G\"ottinger Dissertationen derselben
Verfasserinnen.) auf einem von mir angegebenen Wege bewiesen
worden.''

}
\smallskip

$\bullet$ {\it Rohn}.---Then came Rohn 1911/13 who elaborated
Hilbert's method in much more details. Yet according to
Academician D.\,A. Gudkov (still 1974) this was still not rigorous
enough and required some consideration of dynamical system \`a la
Andronov-Pontryagin ({\it syst\`emes grossiers\/}, alias
structural stability) to become logically robust. The method was
then christened the {\it Hilbert-Rohn method\/}.

What came next? As reported in Gudkov 1974 (p.\,42), a {\it
completely non-rigorous, descriptive attempt\/} of proof (of
Hilbert's Ansatz) was made in Donald 1927 (repeating apparently
the earlier inconclusive attempt of Wright 1907), all expressing
the same methodology as Hilbert. H. Hilton 1936  devoted a paper
to a criticism of Donald's article.

$\bullet$ The real breakthrough occurs with Petrovskii 1933/38 who
supplies a universal inequality valid in all degrees
%bounding
pinching (from both sides) the Euler characteristic  $\chi $ of
the Ragsdale membrane bounding the ovals from inside. One side of
Petrovskii's inequalities reads $\chi\le \frac{3}{2}k(k-1)+1$,
where $k:=m/2$ is the semi-degree of the curve of even order
$m=2k$. This implies Hilbert's Ansatz, and of course much more.
His proof is an explosive cocktail: Euler-Jacobi-Kronecker
interpolation formula combined with Morse theory (1925). For
sextics ($m=6$, hence $k=3$), Petrovskii's upper-bound is $10$ and
so the curve with 11 unnested ovals is ruled out (its $\chi$ being
$11$). Hilbert's Ansatz is (re)proved, or even proved {\it for the
first time\/}, if we accept Gudkov's (sibylline) critiques to both
Hilbert and his students as well as toward Rohn. For this and
other achievements, Petrovskii is often regarded by Arnold as one
%among
of the deepest
%twentieths
20th-century scholar of all Russia.

$\bullet$ Another proof (and perhaps the next one
historiographically) is due (or rather follows) from
Kervaire-Milnor 1961 \cite{Kervaire-Milnor_1961}. In there concise
PNAS-note, this eminent tandem proves what later (or former?) went
known as the {\it Thom conjecture\/} in the special case of
homology classes of degree 3. The Thom conjecture is the assertion
that a smooth oriented surface in the complex projective plane
$\CC P^2$ (the $4$-manifold  of all unordered pairs of points on
the 2-sphere) has
%%%a
genus at least as big as that of an algebraic curve of the
same degree, namely $g=\frac{(m-1)(m-2)}{2}$. This conjecture of
Thom went
%verified
validated by Kronheimer-Mrowka in 1994, but its degree 3 case is
much older (1961 as we just said) and incidentally much based upon
work of the superhero V.\,A. Rohlin (ca. 1951). Now suppose given
a sextic with 11 unnested ovals. Since it is Harnack-maximal the
ovals disconnect the complexification (by Riemann's definition of
the genus). This is a remark of Klein 1876, which naively amounts
to visualize the Galois symmetry of complex conjugation as a
reflecting mirror about a plane leaving invariant a pretzel of
genus $g$ symmetrically sculpted in 3-space and cutting the plane
along $g+1$ ovals (cf. Fig.\,\ref{Pretzel:fig}a for the case
$g=3$). Dissecting one half of the curve gives a bordered surface
which pasted with the ovals-insides creates a surface of genus 0
(topological sphere) whose degree (in the homological sense) is of
course the halved degree of the sextic, namely 3. Rounding corners
(if necessary?) gives a smooth surface whose degree is 3 but of
genus 0 only, hence beating that of a smooth cubic of genus $1$.
Thom's conjecture (i.e. Kervaire-Milnor's theorem) is violated and
%therefore
Hilbert's Ansatz
%is
proved (via pure topology).

$\bullet$ {\it Arnold-Rohlin's era}.---The story does not finish
here, and other
%%%strikingly more elementary proofs
spectacular simplifications of Hilbert's Ansatz came under the pen
of V.\,I. Arnold 1971, and his companion V.\,A. Rohlin 1974.
Arnold 1971 established the congruence $\chi\equiv k^2 \pmod 4$
(valid actually for all dividing curves, not only $M$-curves).
This prohibits the ``Hilbert sextic'' with 11 unnested ovals. In
1974 Rohlin found {\it Rohlin's formula\/} $2(\pi-\eta)=r-k^2$,
where $r$ is the number of ovals while $\pi,\eta$ are resp. the
number of positive and negative pairs of ovals (defined by
comparing orientations induced by the complexification of a
dividing curve with those coming from the bounding annulus for the
nested pair of ovals). This formula implies formally
%, hence is stronger than
Arnold's congruence (compare e.g., Gabard 2013 \cite[p.\,258,
Lemma~26.11]{Gabard_2013_Ahlfors}), and also implies Hilbert's
Ansatz. Indeed in the absence of nesting, the left-side of
Rohlin's formula vanishes and so $r=k^2=3^2=9$, which
%as everybody knows
is not equal to $11$. So Rohlin's formula is the dancing queen of
what can be done in the most elementary way. Its proof involves
capping off the 2 halves of the dividing curve by the bounding
discs of all ovals (hence overlapping violently in case of much
nesting), as to construct two singular $2$-cycles in $\CC P^2$
whose intersection is computed after pushing both objects in
general position.

At this stage nobody cared anymore to prove Hilbert's Ansatz as
the (Arnold-Rohlin) proof was nearly ``{\it from the Book\/}''. Is
the reader at this stage convinced of the truth of Hilbert's
Ansatz just on the basis of what is to be found in our note?
Presumably not as we did not presented any self-contained proof,
but this state of affairs will be remedied  in the sequel of this
text.

What came next? Probably several details but the level of
perfection of Arnold-Rohlin (with slight improvements by Wilson
1978) was so drastic that it left little room for any further
imagination.

$\bullet$ In Jan. 2013, we discovered another little explanation
of Hilbert's Ansatz. Suppose the sextic curve to have $11$
unnested ovals. It seems a reasonable folly to expect that empty
ovals of curves can always be contracted to points (solitary
nodes) via a continuous deformation of the coefficients (and this,
despite the rigidity reputation of algebraic objects). An oval
here is said to be {\it empty\/} if looking inside of it, one sees
no other smaller ovals.
%%inside of it.
Such principles of contractions were actually exploited by Klein
in 1892 (if not earlier) and also form the content of a conjecture
of Itenberg-Viro 1994, which posits that {\it any\/} empty oval of
an algebraic curve can be shrunk to a solitary node. This is a
truly remarkable conjecture which has neither been proved nor been
refuted up to present days.

Let us, cavalier, assume a stronger version of this conjecture
stipulating that {\it all\/} empty ovals can be shrunk {\it
simultaneously\/} (synchronized death of all empty ovals). Apply
this contraction to an unnested $M$-sextic with 11 unnested ovals
to see its underlying Riemann surface of  genus 10 strangulated
into 2 pieces of degree~3 intersecting in 11 points
(Fig.\,\ref{Pretzel:fig}a). But \'Etienne B\'ezout told us a long
time ago that 2 cubics intersect in $3\cdot 3=9$ points. Hilbert's
nesting Ansatz is proved modulo the (unproven) contraction
principle. Alas, the writer does not know if the collective
contraction principle just employed holds true in degree 6, but
this could be quite likely as the shrinking of {\it any\/} single
oval is a result of Itenberg 1994, based on the marvellous
technology of Nikulin 1979 (K3 surfaces, global Torelli for them,
etc.)

\begin{figure}[h]
%\vskip-0.2cm\penalty0
\centering
%\hskip-1.2cm\penalty0
\epsfig{figure=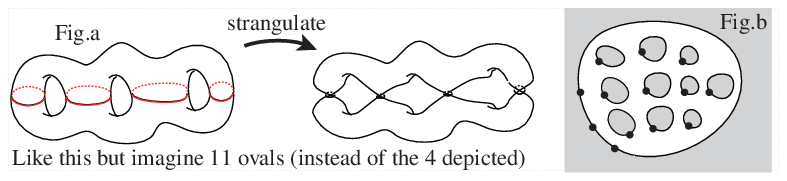,width=122mm} \vskip-5pt\penalty0
  \caption{\label{Pretzel:fig}%
  Contracting all the ovals of an unnested curve toward solitary
  nodes with imaginary-conjugate tangents} \vskip-5pt\penalty0
\end{figure}

%As a philosophical remark,
It may be wondered if our strangulation proof (as heuristic as it
is) was known to Hilbert (or even Klein). It should be remembered
that principles of contraction were often used by Klein (say from
1876 up to 1892) and so it is quite likely that those G\"ottingen
scholars may have thought about this method at least as supplying
some heuristic evidence. (Alas, we know about no trace left in
print.)

$\bullet$ Yesterday evening [19.04.13],
%%%the author (Gabard)
we found another pleasant argument based on the method of total
reality. This method has historical origins in the theory of
conformal mappings, especially Riemann's Nachlass of 1857. Many
subsequent workers were involved in this theorem of Riemann, and
we merely cite them in cascade referring again to Gabard 2013
\cite{Gabard_2013_Ahlfors} for exact references: Schottky 1875/77,
Wirtinger ca. 1900 (unpublished), Enriques-Chisini 1915,
Bieberbach 1925, Grunsky 1937, Courant 1939, Wirtinger 1942
(published this time), Ahlfors 1947/50, A. Mori 1951, \dots,
Huisman 2001, and many others in between.

If this Riemann Nachlass is interpreted extrinsically along the
method used by Harnack 1876 (to prove the after-him called bound
$r\le g+1$ on the number $r$ of ovals), we get a very simple
derivation of Hilbert's Ansatz as we shall explain in the next
Sec.\,\ref{proof-of-Hilbert:sec}. This is the (modest) goal of
this note, but we strongly suspect that when applied more cleverly
Riemann's Nachlass could crack the Ragsdale conjecture or at least
affords simple proofs of the myriad of estimates due to Petrovskii
or his (greatest admirer) Arnold. So our game is an attempt to
shrink back  everything to Riemann via the method of total
reality.

As far as we are concerned, we have to acknowledge some pivotal
inspiration from the paper by Le Touz\'e 2013
\cite{Fiedler-Le-Touzé_2013-Totally-real-pencils-Cubics}, where
the total reality of quintics is explained in a synthetic fashion
(i.e. a Harnack-style argument with boni-intersections gained by
topology or algebra). A simple extension thereof to all curves of
{\it odd degrees\/} is given in Gabard 2013 \cite[Thm~31.12,
p.\,402]{Gabard_2013_Ahlfors}. In that work we failed to treat the
case of even order curves and this is remedied  below
(Sec.\,\ref{Total-reality-even-M-curves-punching-card:sec}) by
showing that total reality is likewise a very simple matter (using
the parity of intersection between ovals).

\section{Proofs}

\subsection{A 2 seconds proof of
Hilbert's Ansatz}\label{proof-of-Hilbert:sec}

\begin{theorem}\label{Hilbert-via-total-reality:thm}
{\rm (Hilbert 1891)}.---A real sextic curve cannot have $11$
unnested ovals.
\end{theorem}

\begin{proof} Inspired by the method of total reality, we consider
a certain ancillary pencil of quartics. Writing down monomials
along increasing degrees
$$
1,\underbrace{x,y}_2,\underbrace{x^2,xy,y^2}_3,\underbrace{x^3,x^2y,
xy^2,y^3}_4, etc.
$$
(best visualized as a pyramid \`a la Newton) we see that quartic
curves depend upon $1+2+3+4+5=\frac{6\cdot 5}{2}=15$ parameters
(the coefficients). Since we are only interested in the equation
up to homothety there are only 14 essential parameters, and so by
linear algebra there can be assigned 13 basepoints  to a pencil of
quartics.

Consider now the 11 ovals of the curve $C_6$ as pigeonholes where
to range the 13 basepoints. Distribute them injectively among the
$11$ ovals while placing the 2 remaining ones on the same oval
(compare Fig.\,\ref{Punching-Card:fig}b if necessary). By a
principle due to M\"obius-von Staudt (and massively used by
Zeuthen 1874 and Harnack 1876) we know that 2 ovals in $\RR P^2$
intersect always in an even number of points (counted by
multiplicity if necessary). Accordingly any curve $C_4$ of our
pencil of quartics
 has one boni-intersection on each oval since we are always
imposing an odd number of basepoints on each oval. So
$13+11=24=4\cdot 6$ real intersections are granted by intersection
theory \`a la M\"obius-von Staudt. This is the maximum permissible
by B\'ezout. We speak of a phenomenon of {\it total reality\/}.

Consider next the (mildly singular) foliation induced by this
pencil of quartics on the inside $R$ of the 11 ovals, which may be
seen as a special case of the {\it Ragsdale membrane\/} bounding
the ovals orientably ``from inside''. Since there is  no nesting
this membrane $R$ is merely a disjoint union of 11 (topological)
discs. It is convenient  to double this membrane to get $2R$, a
union of 11 spheres.

Now 13 basepoints are assigned on the boundary $\partial R=C_6$
but a pencil of quartics has 16 basepoints (B\'ezout once more).
So there is 3 unassigned basepoints, on which we know very little.
In the worst case those 3 points will land inside of the ovals.
Otherwise they can land on the ovals, or eventually outside of
them.

Apply Poincar\'e's index formula 1885 (announced 1881) telling us
that the sum of indices of a foliation\footnote{Poincar\'e 1885
worked this case too, as opposed to simply flows which are
orientable foliations by Ker\'ekj\'art\'o-Whitney 1925/1933.} is
equal to the Euler characteristic of the surface. Each foyer-type
singularity (infinitesimally like the pencil of lines through a
point) has an index of $+1$. In an algebraic foliation, basepoints
induce such foyers and neglecting crudely all singularities of
negative indices gives the estimate:
$$
\chi(2R)=\sum {indices} \le 13+3 \cdot 2=19,
$$
since all the 3 unassigned basepoints contribute for at most 2
foyers (one on each ``face'' of the double)
[Panoramix-double-fax], while the 13 comes of course by doubling
the semi-foyers visible at each of the 13 basepoints assigned on
the $C_6$. On the other hand, $\chi(2R)=2\chi(R)=2\cdot 11=22$.
This is arithmetical nonsense and Hilbert's theorem is proved.
\end{proof}

{\it Historiography}.---Our proof uses {\it Poincar\'e's index
formula\/} (1881/85), of course very well-known to Hilbert (cf.
e.g. the citation by Hellmuth Kneser, one of Hilbert's student, on
the front-page of this note). The (pre)history of Poincar\'e's
formula is probably best recorded in von Dyck 1888, where (vague)
%anticipations
forerunners are listed like Gauss 1839, or Kronecker 1869, and
many others.

\subsection{Rohn's prohibition of the scheme $\frac{10}{1}$
via total reality}

[21.04.13]. Applying the above argument to the dual {\it
non-orientable\/} membrane, say $N$, bounding the ovals from
outside proves Rohn's prohibition (1913 \cite{Rohn_1913}) of the
scheme $\frac{10}{1}$ where $10$ ovals are enveloped in a larger
oval.

\begin{theorem}\label{Rohn-via-total-reality:thm} {\rm (Rohn 1913)}.---An
$M$-sextic curve $C_6$ cannot have $10$ ovals enveloped in a
larger eleventh oval.
\end{theorem}

\begin{proof} As above we consider a total pencil of quartics with
13 basepoints distributed injectively on the 11 ovals safe that
one of the oval absorbs 3 basepoints. Consider the algebraic
foliation induced on the anti-Ragsdale membrane $N$ discussed
above (compare Fig.\,\ref{Pretzel:fig}b).

Applying Poincar\'e's index formula to the doubled membrane $2N$,
we find
$$
\chi(2N)=\sum {indices}\le 13+3\cdot 2=19.
$$
On the other hand $\chi(2N)=2 \chi(N)$, and  $N$ is the union of a
M\"obius band (with $\chi=0$) plus 10 replicas of the 2-sphere
$S^2$ with $\chi=2$, whence $\chi(N)=0+10=10$ and therefore
$\chi(2N)=20$. The proof is complete.
\end{proof}

It is puzzling that those arguments escaped Hilbert and Rohn.
Philosophically it seems that the cause is that those workers were
too much algebraically inclined as opposed to the pure geometry of
Riemann and Poincar\'e. So our argument represents a little
victory of (the angel of) geometry over (the devil of) algebra, as
would say H. Weyl.

More seriously, the 2 proofs given above are fundamental in
completing the programme sketched in the Introd. of Gabard 2013
\cite{Gabard_2013_Ahlfors}. There we explained how the Rohlin-Le
Touz\'e phenomenon of total reality   explains {\it nearly all\/}
prohibitions of Gudkov's census solving Hilbert's 16th in degree
$m=6$. The ``{\it nearly all\/}'' referred precisely to the fact
that this missed the 2 schemes $11$ and $\frac{10}{1}$ prohibited
by Hilbert and Rohn respectively. Since we are now also able to
treat those cases via total reality, we see that in degree $m=6$
the method of total reality is ubiquitous and universal. Of course
we conjecture this to be a general issue for all $m$, compare
again the Introd. of Gabard 2013 \cite{Gabard_2013_Ahlfors}.

\subsection{Total reality of $M$-curves of even order
(the punched card device of Harnack-Le Touz\'e-Gabard)}
\label{Total-reality-even-M-curves-punching-card:sec}

[17.04.13] Our former work (Gabard 2013
\cite{Gabard_2013_Ahlfors}) failed to assess total reality of
$M$-curves of even degree in the strong sense of knowing where to
assign basepoints. (For the weak sense reminiscing perhaps the
{\it Brill-Noetherschen Restsatz\/}, see \cite[Thm~31.8,
p.\,399]{Gabard_2013_Ahlfors}.) Now we show that a very simple
device (already used in Harnack 1876) grants (strong) total
reality of $M$-curves in the even degree case too. (For the odd
degree case see \cite[Thm~31.12, p.\,402]{Gabard_2013_Ahlfors}).

It would cause no trouble to write down directly the general
result and proof but this not the way one usually discovers truth,
so let us work more peacefully. (The pressed reader can directly
move to (\ref{total-reality-M-curves-EVEN-punching-card:thm}).)

Let us start with degree $m=6$ (sextics). Here we look (in
accordance with the general theorem \`a la Brill-Noether
(\cite[Thm~31.8, p.\,399]{Gabard_2013_Ahlfors}) or just
Riemann-Bieberbach) to curves of degree $m-2=4$, i.e. quartics.
Those may be assigned to visit $B=\binom{4+2}{2}-2=13$ basepoints
while still moving in a pencil. On the other hand our $M$-sextic
has $M=11$ ovals. How to distribute basepoints as to ensure total
reality of the quartics-pencil. A priori we may distribute the 13
basepoints on the $11$ ovals (surjectively and in very random
fashion), but then only $22<4\cdot 6=24$ real intersections are
granted. Let us be more specific. Suppose that we distribute
injectively $11$ basepoints on the $11$ ovals, while placing the 2
remaining points on 2 distinct ovals (cf. the black dots on
Fig.\,\ref{Punching-Card:fig}a). Then a $C_4$ of the pencil has
$2\cdot 2+ 9\cdot 2=22$ real intersections granted (compare again
Fig.\,\ref{Punching-Card:fig}a, where the white dots are extra
intersections gained for parity reasons of the intersection of two
ovals=even degree circuits in older jargon). This is not enough
for total reality to be valid at $24=4\cdot 6$. If however our
``punching-card machine'' assigns the 2 additional points on the
same oval (like on Fig.\,\ref{Punching-Card:fig}b), then we get
$1\cdot 4+10\cdot 2=24$ real intersections and total reality is
demonstrated (compare again Fig.\,\ref{Punching-Card:fig}b
counting now also the bonus intersections materialized by  white
bullets). Again, we used the classical fact that 2 ovals have an
even number of intersections counted by multiplicity.

\begin{figure}[h]
%\vskip-0.2cm\penalty0
\centering
%\hskip-1.2cm\penalty0
\epsfig{figure=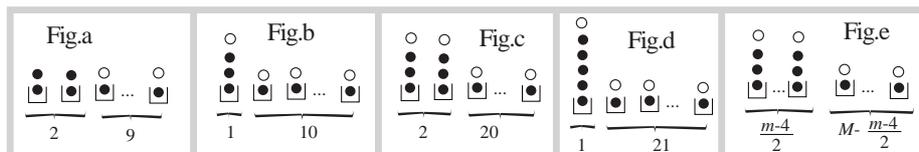,width=122mm} \vskip-5pt\penalty0
  \caption{\label{Punching-Card:fig}%
  Repartition of basepoints in black, and
  extra boni-intersections gained by
  the punching-card trick in white} \vskip-5pt\penalty0
\end{figure}

Optionally, it is a pleasant exercise to check that the same trick
and (little) miracle of boni works in degree 8.
%(Yet beware not
%becoming a capitalist!)
Now since we expect a pencil of sextics we have
$B=\binom{8}{2}-2=26$ basepoints available, while we have
Harnack's bound many, i.e., $M_8=g+1=\frac{7\cdot 6}{2}+1=22$,
 ovals at disposal. Again to create as much boni intersections
as possible, it is valuable to disperse the 4 excess/surplus
basepoints as 2 groups of height 2 (Fig.\,c). Then by the evenness
principle for intersecting ovals we have $2\cdot 4+20 \cdot
2=48=6\cdot 8$ real intersections granted (the maximum permissible
by B\'ezout). Total reality is proved. It may be noted that
choosing a distribution like Fig.\,d, where the 4 extra {\it
bases/eyes}(=abridged for basepoints) are concentrated on a single
oval, total reality is likewise granted (as $1\cdot 6+21\cdot
2=48=6\cdot 8$). More basically, without arithmetics, we may infer
this by noticing that as the degree is odd in restriction to each
pigeonhole materializing an oval, we gain also one white-bullet
above all pigeonholes as in the former case, whence total reality.

As customary in such games, it is straightforward to extend to any
degrees and we arrive at the following result (in philosophical
substance known to Harnack 1876, or Enriques-Chisini 1915, or Le
Touz\'e 2013, and of course many others like Joe Harris, Johan
Huisman, etc.):

\begin{theorem}\label{total-reality-M-curves-EVEN-punching-card:thm}
Given any $M$-curve of even degree $m$, the pencil of curves of
order $m-2$ assigned to visit a
%distribution
repartition of basepoints having odd ``degree'' on each oval is
totally real. Further, and in accordance with the
Riemann-Schottky-Enriques-Bieberbach theorem, the pencil possesses
exactly one mobile point circulating along each real circuit.
\end{theorem}

\begin{proof} The number of basepoints
%%%(=eyes)
for a pencil of $(m-2)$-tics is
$B=\binom{(m-2)+2}{2}-2=\binom{m}{2}-2$. Harnack's bound for the
given $m$-tics $C_m$ is $M=\binom{m-1}{2}+1$. Hence the excess of
basepoints over the number of ovals is
$B-M=[1+2+\dots+(m-1)]-2-[1+2+\dots+(m-2)]-1=m-4$.

We may for instance
%%%%%repartee (=r\'epartir in French?)
share out
%CHECKED IN DICO
%
those $m-4$ extra
%%%%eyes
basepoints as on Fig.\ref{Punching-Card:fig}\,e, i.e. by splitting
them in $\frac{m-4}{2}$ groups of ``height'' 2. This is
arithmetically meaningful as $m$ is supposed even. Then, on
counting real intersections forced by intersection theory (of
ovals), we are granted of (cf. again Fig.\,e)
$4(\frac{m-4}{2})+2[M-\frac{m-4}{2}]=2(m-4)+2M-(m-4)=2M+(m-4)
=2[\frac{(m-1)(m-2)}{2}+1]+(m-4)=(m-1)(m-2)+(m-2)=m(m-2)$, and
total reality is demonstrated.

Of course as one bonus intersection is gained on each oval the
count works whenever the repartition has odd degree on each oval
and the asserted total reality is established.

The last clause of the statement (analogy with Riemann {\it et
al.\/}) follows by noticing that the boni intersections (white
bullets on Fig.\,\ref{Punching-Card:fig}e) are unique on each
oval.
\end{proof}

To what is this (theorem) useful at all? Always keep in mind that
we are geometers not (so much) obnubilated by the magics of
arithmetics. What could be desired is an intelligence  capable of
visualizing such pencils, and playing maybe with the
Poincar\'e-von Dyck index formula (1885/88 respectively). (Recall
that Ragsdale 1906 cites von Dyck 1888 precisely for this purpose,
yet not so surprising as  both are docile students of Klein.)
Granting this visualization (or just an arithmetical/combinatorial
corollary of it like Poincar\'e's index formula) one could maybe
infer a new proof of Hilbert's prohibition of the (unnested)
scheme $11$. [{\it Added in proof\/} [20.04.13]: This is indeed
possible see (\ref{Hilbert-via-total-reality:thm}).]

Even more ambitiously, we could dream that a thorough inspection
of such pencils (maybe combined with the Arnold/Rohlin tricks of
splitting and closing Klein's orthosymmetric half by suitable
limbs of real membranes) could imply a proof of the elusive {\it
Ragsdale conjecture\/} $\chi \le k^2$. The so-called {\it Arnold
surface\/} (i.e. Klein's half glued with Ragsdale's membrane) is
always embedded but alas not ever orientable (else Ragsdale would
be a trivial consequence of Thom, cf.
\cite[Sec.\,33]{Gabard_2013_Ahlfors}). By-standing to the former
there is also (what we propose to call) {\it Rohlin's surface\/}
which has the advantage of orientableness, but a ``singular
chain'' (in the jargon of Lefschetz-Alexander-Eilenberg) which is
not embedded in general. For the definition of this surface it
suffices us to say that it is the one involved in the proof of
Rohlin's formula.
%%%%(\ref{Rohlin-formula:thm}).
One could try to
inspect the intersection of (say) an imaginary member of the
pencil with the semi-Riemann surfaces of Arnold or better Rohlin
(which is orientable hence defining an integral homology class).
Further, imaginariness amounts to unilaterality (in the sense of
Gabard 2006) under total reality: this is just to say that all
intersections have to be located in the same half (whence our
pompous name ``unilaterality'') apart from those coming from the
assigned real basepoints. We admit that this approach to Ragsdale
is probably overoptimistic, yet it seems wise to leave open any
possible strategy toward the elusive conjecture. For another
surely much more mature attack, cf. Fiedler's programme sketched
in his terrible letter dated [14.03.13] reproduced below.

\smallskip
$\bullet\bullet\bullet$ [14.04.13] dimanche 14 avril 2013
10:33:08, Fiedler wrote:

Dear Alexandre,

just to let you know some old results, which could perhaps be
worth to be explored in more generality.

In the mid 90' I have tried to prove the Ragsdale conjecture for
$M$-curves with an idea coming from knot theory: bring the object
first to its most symmetric position.

DEFINITION: A real curve $X$ is said to be symmetric if it is
invariant under a (non-trivial) holomorphic involution $s$ of the
complex projective plan.

The idea was to deform first the curve into a symetric one and
then to explore the additional information coming from pencils of
lines which are real simultaneously for both real structures on
$X$, namely $conj$ and $s.conj=conj.s$.

It has failed miserably, because I have proven the following
theorem.

THEOREM 1. For a symmetric $M$-curve of degree $2k$ the following
refinement of the Gudkov-Rokhlin congruence holds:
$$
p - n = k^2 mod 16.
$$
So, roughly half of the $M$-curves are not symmetric. It is
amazing that some conjectures are false in general but true for
symmetric $M$-curves. Hilbert's conjecture: because the Gudkov
curve is not symetric. Viro's conjecture: because an $M$-curve of
degree $8$ with a nest of depth $3$, which has an odd number of
innermost ovals, can not be symetric. On the other hand the idea
was not soo bad because I have proven the following theorem.

THEOREM 2. If a symmetric $M$-curve of degree $2k$ has a nest of
depth $k-2$ then $p - n <= k^2$ and if equality holds then the
Arnold surface $A^+$ is orientable.

These results were never published, because I was already to much
into knot theory. But you can find some information about it in a
paper of Erwan Brugalle and in a paper of my student Sebastien
Trilles. Very best, Thomas

\section{Speculations}

\subsection{Flirting with Miss Ragsdale}

[19.04.13] As we noticed in the previous section there is some
chance that total reality can crack the still open (and elusive)
Ragsdale conjecture for $M$-curves of even degrees, namely the
estimate $\chi \le k^2$. (NB: the full-Ragsdale conjecture posits
$\vert \chi\vert \le k^2$.)

Our idea is based on the previous total reality phenomenon
(\ref{total-reality-M-curves-EVEN-punching-card:thm}) for plane
$M$-curves of even order $m=2k$. This phenomenon is merely a
Harnack-style argument with boni-intersections gained by evenness
of the intersection of 2 ovals, yet it looks so robust and easy
that it seems a reasonable attack upon Ragsdale. To add some
spiciness to our strategy it should always be remembered that
total reality truly belongs to Ahlfors 1950 (or maybe Klein
according to Teichm\"uller 1941), yet on the case at hand of
$M$-curves it is really due to Riemann's Nachlass of 1857. This
being said we expect a big flirt between Riemann 1857 and Miss
Ragsdale 1906.

Now the idea would be that some intelligence able to visualize
properly this pencil (while extracting the relevant combinatorial
aspects) should be able to derive from the total reality of such
pencils the estimate $\chi\le k^2$ (and perhaps its opposite
$-k^2\le \chi$ too). As we said the trick could be to intersect an
imaginary (hence unilateral) member of the pencil with the {\it
Rohlin surface\/} obtained by aggregating the bounding discs of
all ovals. The little technical difficulty is that one requires to
understand the intersection indices so obtained. This game still
escapes me slightly but  is well understood by Arnold and Rohlin,
compare e.g. the proof of Rohlin's formula.
%%%(\ref{Rohlin-formula:thm}).
One of the additional difficulty is
that the pencil of $(m-2)$-tics will have non-assigned basepoints
and those create additional intersections somewhat harder to
control since their location is not known a priori (in contrast to
the assigned basepoints). We hope to discuss this issue in more
detail later.

Perhaps another general philosophical comment. As we emphasized
the method of total reality used in
(\ref{total-reality-M-curves-EVEN-punching-card:thm}) for
$M$-curves of even order is just an avatar of a Harnack-style
argument. Historically it may also be remembered that the
prototype for this sort of reasonings goes back to
%the Danish
%superstar
Zeuthen 1874, who impressed much Klein and so indirectly Harnack.
Of course Zeuthen himself refers back to M\"obius and von Staudt
who expressed in modern vocabulary fixed what we call nowadays the
intersection theory of $\RR P^2$. So our strategy toward Ragsdale
bears some close analogy with Harnack's synthetic proof of the
so-called Harnack inequality $r\le g+1$. Hence if Ragsdale
estimate $\chi \le k^2$ (or its general version with absolute
value $\vert \chi \vert \le k^2$) is correct, it is likely that
its proof proceeds along a similar line than that of Harnack's
inequality which is so-to-speak the most fundamental estimate for
the topology of real curves. Maybe this vague analogy gives
another weak evidence that we are on the right track toward
proving Ragsdale [or related results \`a la Petrovskii-Arnold].

A last philosophical remark is in order. As we all know Klein 1876
offered a somewhat more conceptual (or intrinsic) justification of
$r\le g+1$ by using merely topology, as opposed to the synthetic
geometry of Harnack. As a rule Klein's argument is conceptually
somewhat more limpid than Harnack's which is a bit tricky
arithmetics/cominatorics. Accordingly one could also suspect a
topological proof of Ragsdale somewhat easier than via total
pencils. In substance this could be our crude but erroneous
approach via Thom's genus (lower) bound
(\cite[Sec.\,33]{Gabard_2013_Ahlfors}) or the programme sketched
by Fiedler using knot theory. Notwithstanding we may expect that
our synthetical strategy has still some good chance to crack
Ragsdale, and we hope being able to attack this question in the
future.

\subsection{A disappointing estimate with zero-information
on the unassigned basepoints}

[19.04.13, but TeXified 21.04.13] If we ape directly the proof of
Theorem~\ref{Hilbert-via-total-reality:thm}, i.e. Hilbert's Ansatz
via Riemann's Nachlass and Poincar\'e's index formula to the case
of a general even degree $m=2k$ $M$-curve, we get an estimate
which is extremely disappointing when $m> 6$. This will be exposed
right below, and the dream should be to get a sharper estimate by
trying to control better the location of the unassigned
basepoints.

Let $C_{m}$ be a plane  $M$-curve of even degree $m=2k$. By
Theorem~\ref{total-reality-M-curves-EVEN-punching-card:thm} we
have a phenomenon of total reality for a pencil of curves of
degree $(m-2)$ assigned to visit $B=M+(m-4)$ basepoints (cf. 1st
paragraph of its proof).

So applying Poincar\'e's index formula to the (doubled) Ragsdale
membrane, $2R$, we get the following bound after noting that there
are $(m-2)^2-B$ unassigned basepoints:
\begin{align*}
\chi(2R)=\sum {indices}&\le B+2[(m-2)^2-B]\cr
&= 2(m-2)^2-B \cr
&= 2(m-2)^2-M-(m-4) \cr
&= 2(m-2)^2-\frac{(m-1)(m-2)}{2}-1-(m-4) \cr
&= 2(2k-2)^2-(2k-1)(k-1)-1-(2k-4) \cr
&= 8(k-1)^2-(2k-1)(k-1)-1-2(k-2) \cr
&= 8(k-1)^2-(2k-1)(k-1)-2(k-1)+1 \cr
&= (k-1)[8(k-1)-(2k-1)-2]+1 \cr
&= (k-1)[6k-9]+1 \cr
&= [(k-1)3(2k-3)]+1.
\end{align*}
Therefore
$$
\chi=\chi(R)\le \frac{3}{2}(k-1)(2k-3)+\frac{1}{2}.
$$
While this is interesting for $m=6$ (as we saw), for $m=8$ (so
$k=4$) this bound is useless, yielding only $\chi\le \frac{3}{2}
3\cdot 5+\frac{1}{2}=\frac{45}{2}+\frac{1}{2}=23$, which is stupid
as by Harnack we know $\chi\le M=22$. Of course Petrovskii bound
is even much better yielding $\chi\le
\frac{3}{2}k(k-1)+1=\frac{3}{2}4\cdot 3+1= 19$. Further
asymptotically our bound is $\approx 3 k^2$ which is completely
useless in comparison to Harnack's bound $\chi\le M\approx 2 k^2$

So we see that our method needs to be refined and there is of
course much maneuvring room to do this, e.g. taking into account
singularity of negatives indices and/or trying to predict the
location of the unassigned basepoints. (It is perhaps here that
deep predestination process of algebraic geometry \`a la
Euler-Cayley-Bacharach or the Euler-Jacobi-Kronecker interpolation
formula) have to enter into the scene. At this stage, it is safe
to leave the topic to other more qualified workers (the dream
being to crack Ragsdale's conjecture, and more modestly to reprove
Petrovskii or the strengthened version  due to Arnold!)

As a last loose idea it is important that the phenomenon of total
reality always implies one point circulating on each circuit
(=ovals). By virtue of the holomorphic character of the
Riemann-Ahlfors map this gives raise to a dextrogyration, i.e.
points moves compatibly with the complex orientation induced on
the ovals (Rohlin's jargon). This and other dynamical principles
should perhaps aid to predict the location of the unassigned base
points.

\medskip

{\small {\bf Acknowledgements.} It is a great pleasure to thank
S\'everine Fiedler-Le Touz\'e and Thomas Fiedler for direct
inspiration upon this little work, as well as Oleg Viro,
Viatcheslav Kharlamov, Alexis Marin, Eugenii Shustin, Stepa
Orevkov, for exceptionally instructive e-mails reproduced in
Gabard 2013 \cite{Gabard_2013_Ahlfors}. }

{\small
%%%%%%%%%%%%%%%%%%%SAXO%%%%%%%NEW BIBLIO

}

{
\hspace{+5mm} % To get a little bit of space between the figures
{\footnotesize
\begin{minipage}[b]{0.6\linewidth} Alexandre
Gabard

Universit\'e de Gen\`eve

Section de Math\'ematiques

2-4 rue du Li\`evre, CP 64

CH-1211 Gen\`eve 4

Switzerland

alexandregabard@hotmail.com
\end{minipage}
\hspace{-25mm} }

\end{document}